\documentclass[12pt,reqno,a4paper]{amsart}

\usepackage{latexsym, enumerate}
\usepackage{amssymb,amsmath,mathtools}
\usepackage[hidelinks]{hyperref}

\mathtoolsset{showonlyrefs}

\let\coloneq\relax 
\newcommand{\coloneq}{\coloneqq}

\newcommand{\rd}{{\rm d}}
\newcommand{\e}{{\rm e}}
\newcommand{\Ran}{\mathop{\rm Ran}}
\newcommand{\N}{{\mathbb N}}
\newcommand{\R}{{\mathbb R}}
\newcommand{\C}{{\mathbb C}}

\newcommand\eps{\epsilon}

\newcommand\I{\mathrm{i}}

\DeclareMathOperator{\tr}{Tr}
\DeclareMathOperator{\supp}{supp}
\newcommand{\g}{{\rm g}}

\newtheorem{theorem}{Theorem}[section]

\newtheorem{corollary}[theorem]{Corollary}
\newtheorem{lemma}[theorem]{Lemma}
\newtheorem{remark}[theorem]{Remark}

\begin{document}

\title[{Orthonormal Spectral Cluster Bounds in Nonpositive Curvature}]{Orthonormal Spectral Cluster Bounds on Manifolds with Nonpositive Curvature}

\author{Jean-Claude Cuenin}
\address[Jean-Claude Cuenin]{Department of Mathematical Sciences, Loughborough University, Loughborough, Leicestershire, LE11 3TU United Kingdom}
\email{J.Cuenin@lboro.ac.uk}

\author{Ngoc Nhi Nguyen}
\address[Ngoc Nhi Nguyen]{Laboratoire Paul Painlevé UMR 8524 et équipe projet INRIA PARADYSE, Université de Lille, Cité Scientifique, F-59655 Villeneuve d'Ascq Cedex, France}
\email{ngoc-nhi.nguyen@univ-lille.fr}

\author{Xiaoyan Su}
\address[Xiaoyan Su]{Department of Mathematical Sciences, Loughborough University, Loughborough, Leicestershire, LE11 3TU United Kingdom}
\email{X.Su2@lboro.ac.uk}

\date{May 2026}

\thanks{\copyright\, 2026 by the authors. This paper may be reproduced, in its entirety, for non-commercial purposes.}

\begin{abstract}
Let \((M,g)\) be a closed \(n\)-dimensional Riemannian manifold with nonpositive sectional curvature. We prove sharp, logarithmically improved spectral cluster bounds for orthonormal systems in the supercritical range. More precisely, for spectral windows of size \((\log \lambda)^{-1}\), we obtain the orthonormal analogue of the logarithmically improved \(L^q\) estimates of Hassell--Tacy. Our argument combines the universal orthonormal spectral cluster bounds of Frank--Sabin with Bérard-type kernel estimates and a generalization of the Bourgain--Shao--Sogge--Yao multiplier estimate to the orthonormal setting.
\end{abstract}


\maketitle

\section{Introduction}

\subsection{Universal bounds}
Let $(M, \g)$ be a  closed $n$-dimensional manifold and $-\Delta_\g $ be the associated Laplace--Beltrami operator, for $n\geq 2$.
For $\lambda\geq 2$ and $\varepsilon(\lambda)\in (0,1]$, the spectral projector with window size $\varepsilon$ and its corresponding eigenspace are defined, using the spectral theorem, as
\[
  \Pi_{\lambda,\varepsilon(\lambda)}\coloneq \mathbf{1}(\sqrt{-\Delta_\g}\in[\lambda,\lambda+\varepsilon(\lambda))), \quad E_{\lambda,\varepsilon(\lambda)}\coloneq\Ran\Pi_{\lambda,\varepsilon(\lambda)} .
\]
For $\varepsilon(\lambda)=1$, Sogge \cite{Sogge-1988} proved the universal bounds
\begin{align}\label{ineq:Sogge-uni-bd}
  \|u\|_{L^q(M)} \lesssim\lambda^{\delta(q)}\|u\|_{L^{2}(M)} ,\quad  
  u\in E_{\lambda,1}
\end{align}
where, for $2\leq q\leq\infty$,
\begin{align}\label{eq:delta(q)}
  \delta(q)\coloneq  \begin{cases}
    \frac{n-1}{2}(\frac{1}{2}-\frac{1}{q}), & 2\leq q\leq q_c ,\\
    n(\frac{1}{2} -\frac{1}{q})-\frac{1}{2},  & q_c\leq q \leq \infty ,
  \end{cases}
  \quad \text{and} \quad
  q_c\coloneq \frac{2(n+1)}{n-1} .
\end{align}
As was shown in \cite{MR3645429}, the universal  bounds \eqref{ineq:Sogge-uni-bd} are sharp on
any closed manifold.

More recently, Frank--Sabin \cite[Theorem 2]{FrankSabin-2017clust} proved the following generalization of \eqref{ineq:Sogge-uni-bd}: for any orthonormal system $(u_j)_{j\in J}\subset E_{\lambda,1}$ and coefficients $\nu=(\nu_j)_{j\in J}\subset\mathbb{C}$,
\begin{align}\label{ineq:FrankSabin-uni-bd}
  \bigg\|\sum_{j\in J}\nu_j |u_j|^2\bigg\|_{L^{q/2}(M)}\lesssim  \lambda^{2\delta(q)} \|\nu\|_{\ell^{\alpha(q)}} ,
\end{align}
where the implicit constant is independent of $\lambda$, $J$, $\nu$ and the orthonormal system, and
\begin{align}\label{eq:alpha-q}
  \alpha(q)\coloneq\begin{cases}
    \frac{2q}{q+2}& 2\leq q\leq q_c ,\\
    \frac{q(n-1)}{2n},  & q_c\leq q \leq \infty .
  \end{cases}
\end{align}
Here and in the following, $J$ denotes any countable index set. By contrast, combining \eqref{ineq:Sogge-uni-bd} with the triangle inequality yields \eqref{ineq:FrankSabin-uni-bd} with \(\alpha = 1\), even without the orthogonality assumption. The key point is that, for $q>2$, Frank--Sabin's inequality holds with some \(\alpha > 1\). This is in line with the Lieb--Thirring philosophy that orthogonality needs space.

\subsection{Logarithmically improved bounds}
We now assume that $(M,g)$ has nonpositive curvature, in the sense that all the sectional curvatures are nonpositive. In this setting, for $\varepsilon(\lambda)=(\log \lambda)^{-1}$, Hassell and Tacy \cite{HassellTacy-2015} obtained the sharp bounds 
\begin{align}\label{eq:Hassell-Tacy}
  \|u\|_{L^{q}(M)} \lesssim_q (\log \lambda)^{-1/2} \lambda^{n(\frac{1}{2}-\frac{1}{q})-\frac{1}{2}}\|u\|_{L^2(M)},\quad u\in E_{\lambda, (\log \lambda)^{-1}}
\end{align}
for all \emph{supercritical} exponents $q>q_c$.  A well-known orthogonality argument shows that \eqref{eq:Hassell-Tacy} cannot be improved.

The case $p=\infty$ in \eqref{eq:Hassell-Tacy} follows from earlier results of B\'erard \cite{Berard-1977}
and is related to improved remainder estimates in the Weyl law,
\begin{align}\label{eq:weyl-law}
  N(\lambda)\coloneq\operatorname{Tr}(\mathbf{1}({-\Delta_\g }<\lambda^2))=|M|\frac{|B^n|}{(2\pi)^n}\lambda^{n}+\mathcal{O}\left(\frac{\lambda^{n-1}}{\log\lambda}\right).
\end{align}
The whole range of supercritical $L^q$ bounds was later established under a weaker dynamical assumption in the work of Canzani and Galkowski \cite{CanzaniGalkowski-2023-growth}.

The aim of this paper is to generalize \eqref{eq:Hassell-Tacy} to systems of orthonormal functions. The following theorem is our main result.

\begin{theorem}\label{thm:main}
  Let $(M,g)$ be a closed Riemannian manifold of dimension $n$, all of whose sectional curvatures are nonpositive, and let $q>q_c$ and $1\leq \beta<\frac{q(n-1)}{2n}$. Then for any orthonormal system $(u_j)_{j\in J}\subset E_{\lambda,(\log\lambda)^{-1}}$ and coefficients $\nu=(\nu_j)_{j\in J}\subset\C$,
  \begin{align}\label{ineq:main}
     \bigg\|\sum_{j\in J}\nu_j |u_j|^2\bigg\|_{L^{q/2}(M)}\lesssim_{q,\beta} (\log\lambda)^{-1} \lambda^{n(1-\frac{2}{q})-1} \|\nu\|_{\ell^{\beta}}.
  \end{align}
\end{theorem}

The theorem is optimal in the following two senses:
first, the factor $(\log\lambda)^{-1} \lambda^{n(1-\frac{2}{q})-1}$ cannot be improved. This follows from the optimality of~ \eqref{eq:Hassell-Tacy} by taking $\# J=1$.
Second, if $q<\infty$, the inequality fails for all $\beta\ge  \frac{q(n-1)}{2n}$. This is seen by taking an orthonormal basis $(u_j)_{j\in J}$ of the eigenspace $E_{\lambda,(\log\lambda)^{-1}}$ and $\nu_j=1$ for all $j\in J$. Indeed, by H\"older's inequality,
\begin{align*}
&|M|^{1-\frac{2}{q}}\limsup_{\lambda \to \infty}  \lambda^{1-n} (\log \lambda)   \bigg\|\sum_{j\in J} |u_j|^2\bigg\|_{L^{q/2}(M)}
\geq 
\limsup_{\lambda \to \infty}  \lambda^{1-n} (\log \lambda)   \bigg\lVert\sum_{j\in J} |u_j|^2\bigg\rVert_{L^{1}(M)}
\\
&=\limsup_{\lambda \to \infty} \lambda^{1-n} (\log \lambda)\dim E_{\lambda,(\log\lambda)^{-1}}
=\limsup_{\lambda \to \infty} \lambda^{1-n} (\log \lambda)(N(\lambda+(\log\lambda)^{-1})-N(\lambda))>0
\end{align*}
where the final inequality follows from the elementary fact that $\limsup_{\lambda \to \infty}  \lambda^{-n}N(\lambda)>0$.
On the other hand, $\|\nu\|_{\ell^{\beta}}=(\# J)^{1/\beta}$.
Combining these observations, we see that \eqref{ineq:main} can only hold if
\begin{align}
  \frac{\lambda^{n-1}}{\log\lambda}\lesssim\lambda^{n(1-\frac{2}{q})-1}(\log\lambda)^{-1 }\Big(\frac{\lambda^{n-1}}{\log\lambda}\Big)^{\frac{1}{\beta}}.
\end{align}
It follows from the relation $n(1-\frac{2}{q})-1 + (n-1)\frac{2n}{q(n-1)} = n-1$ that we must have $\beta < \frac{q(n-1)}{2n}$ if \(q<\infty\). For \(q=\infty\), \eqref{ineq:main} holds with $\beta=\infty$. This was proved by  Ren and Zhang \cite{RenZhang-2024}. Alternatively, it follows from \eqref{eq:Hassell-Tacy} and, e.g., \cite[Proposition 2.1]{CueninNguyenSu2025}.

\subsection*{Notation} The notation $A\lesssim B$ means $A\leq CB$, for some unspecified
constant $C$ that may depend on fixed quantities such as $M,q$ and $\beta$ but not on $\lambda,J,\nu$ or the orthonormal system (or on $W,W_1,W_2$ in the dual estimates). To emphasize the dependence on parameters, we sometimes use subscripts, e.g. $A\lesssim_{q,\beta} B$. When we write $A\lesssim_N B_N$, we mean that for every $N>0$ there exists a constant $C_N$ such that $A\leq C_N B_N$.

%
For an operator $S:L^2(M)\to L^2(M)$, we denote its operator norm by $\|S\|$ and its $p$-Schatten norm by $\|S\|_{\mathfrak{S}^p}\coloneqq (\tr(S^*S)^{\frac{p}{2}})^{\frac{1}{p}}$ for $1\leq p<\infty$. We often omit $M$ from the notation, e.g. we write $L^p$ instead of $L^p(M)$.


\section{Spectral multiplier theorem for orthonormal systems}\label{sec:multiplier}

In this section, we prove a spectral multiplier estimate that only uses the universal bounds of Frank--Sabin as input. This is a generalization of the Bourgain--Shao--Sogge--Yao multiplier estimate \cite[Lemma 2.3]{BourgainShaoYaoSogge-2015} to the orthonormal setting and may be of independent interest.

We set
\begin{align*}
P=\sqrt{-\Delta_\g}\quad \mbox{and}\quad
  \Pi_{k,1}\coloneq\mathbf{1}(P\in [k,k+1))\ \mbox{for }k\in \N\cup\{0\}.
\end{align*}

\begin{theorem}\label{lemma multiplier}
  Let $(M,\g)$ be a closed Riemannian manifold of dimension $n\geq 2$, let $m:[0,\infty)\to \C$ be a bounded Borel function and $2\leq q\leq\infty$. Let $\delta(q)$ and $\alpha(q)$ be given by \eqref{eq:delta(q)} and \eqref{eq:alpha-q}, respectively. Assume that
  \begin{align*}
    K_q(m)\coloneq\bigg(\sum_{k=0}^{\infty}\sup_{\tau\in [k,k+1)}|m(\tau)|^2(1+k)^{2\delta(q)}\bigg)^{1/2}<\infty .
  \end{align*}
  Then for any orthonormal system $(u_j)_{j\in J}\subset L^2(M)$ and coefficients $\nu=(\nu_j)_{j\in J}\subset\C$,
  \begin{align}\label{eq.multiplier 1}
    \bigg\|\sum_{j\in J}\nu_j |m(P)u_j|^2\bigg\|_{L^{q/2}(M)}\lesssim K_q(m)^2\|\nu\|_{\ell^{\alpha(q)}}.
  \end{align}
  Moreover, for all $W_1,W_2\in L^{2(q/2)'}(M)$,
  \begin{align}\label{eq.multiplier 2}
\|W_1m(P)W_2\|_{\mathfrak{S}^{\alpha(q)'}(L^2(M))}\lesssim K_q(|m|^{1/2})^2 \|W_1\|_{L^{2(q/2)'}(M)}\|W_2\|_{L^{2(q/2)'}(M)}.
  \end{align}
\end{theorem}

\begin{proof}
  By Frank--Sabin's duality principle \cite[Lemma 3]{FrankSabin-2017rest},~\eqref{eq.multiplier 1} is equivalent to
  \begin{align}\label{spectral multiplier bound dual version}
    \|Wm(P)\|^2_{\mathfrak{S}^{2\alpha(q)'}}\lesssim K_q(m)^2\|W\|^2_{L^{2(q/2)'}} ,\quad W\in L^{2(q/2)'}.
  \end{align}
  Assume first that $\supp m\subset [0,N+1]$ for some $N\in\N$. For $\varepsilon>0$, we define
  \begin{align}
S_\varepsilon\coloneq\sum_{k=0}^Nc_{k,\varepsilon}^{-1}\Pi_{k,1}m(P),\quad c_{k,\varepsilon}\coloneq\sup_{\tau\in [k,k+1)}|m(\tau)|+\varepsilon.
  \end{align}
  Using orthogonality,
  \begin{align}\label{orthogonality}
    \Pi_{k,1}\Pi_{k',1}=\delta_{k,k'}\Pi_{k,1} ,\quad k,k'\in\N ,
  \end{align}
  and the support assumption on $m$,
  we observe that
  \begin{align}
    \sum_{k=0}^Nc_{k,\varepsilon}   \Pi_{k,1} S_\varepsilon=\sum_{k=0}^N\Pi_{k,1}m(P)=m(P).
  \end{align}
  By the spectral theorem,  
  \begin{align}\label{eq: norm S leq 1}
\|S_\varepsilon\|_{L^2\to L^2}\leq 1.    
  \end{align}
  Thus
  \begin{align*}
    \|Wm(P)\|^2_{\mathfrak{S}^{2\alpha(q)'}}
    &\leq \bigg\|\sum_{k=0}^Nc_{k,\varepsilon}   W\Pi_{k,1}\bigg\|^2_{\mathfrak{S}^{2\alpha(q)'}}
    =\bigg\|\sum_{k=0}^Nc_{k,\varepsilon}^2   W\Pi_{k,1}\overline{W}\bigg\|_{\mathfrak{S}^{\alpha(q)'}}
    \leq \sum_{k=0}^{\infty}c_{k,\varepsilon}^2\|W\Pi_{k,1}\overline{W}\|_{\mathfrak{S}^{\alpha(q)'}}\\
    &\lesssim \Big( \sum_{k=0}^{\infty}\sup_{\tau\in[k,k+1)}|m(\tau)|^2(1+k)^{2\delta(q)}+ \varepsilon^2 N^{2\delta(q)+1} \Big) \|W\|^2_{L^{2(q/2)'}}
  \end{align*}
where we used $\|AB\|_{\mathfrak S^p}\leq \|A\|_{\mathfrak S^p}\|B\|$ and \eqref{eq: norm S leq 1} in the first inequality, followed by
\(\|A\|_{\mathfrak S^{2p}}^2=\|AA^*\|_{\mathfrak S^p}\) and orthogonality of spectral projectors to eliminate cross terms. The final inequality
uses the Frank--Sabin bound \eqref{ineq:FrankSabin-uni-bd} for spectral projectors, or more precisely, the dual inequality
\[
\|W\Pi_{k,1}\overline{W}\|_{\mathfrak{S}^{\alpha(q)'}}\lesssim (1+k)^{2\delta(q)}\|W\|^2_{L^{2(q/2)'}},
\]
see \cite[(16)]{FrankSabin-2017clust}. Passing to the limit $\varepsilon\to 0$ yields ~\eqref{spectral multiplier bound dual version}. 

To remove the support assumption, observe that, since $P$ is a self-adjoint operator with discrete spectrum, $\mathbf{1}_{[0,N+1]}(P)$ is a sequence of monotonically increasing finite-dimensional orthogonal projections that converges strongly to the identity operator on $L^2(M)$, as $N\to\infty.$ Setting $m_N(P):=m(P)\mathbf{1}_{[0,N+1]}(P)$, it follows by \cite[Theorem 5.2]{MR246142} that
\[
\|Wm(P)\|^2_{\mathfrak{S}^{2\alpha(q)'}}\leq \sup_N\|Wm_N(P)\|^2_{\mathfrak{S}^{2\alpha(q)'}}
\lesssim K_q(m_N)^2\|W\|^2_{L^{2(q/2)'}}\leq K_q(m)^2\|W\|^2_{L^{2(q/2)'}}.
\]
  
The bound~\eqref{eq.multiplier 2} follows from~\eqref{spectral multiplier bound dual version} and its dual, applied to $|m|^{1/2}$, since
  \begin{align*}
\|W_1m(P)W_2\|_{\mathfrak{S}^{\alpha(q)'}}
    &\leq 
    \left\lVert W_1|m|^{1/2}(P)\right\rVert_{\mathfrak{S}^{2\alpha(q)'}}
    \bigg \lVert\frac{m}{|m|}(P)\bigg \rVert
    \left\lVert |m|^{1/2}(P)W_2\right\lVert_{\mathfrak{S}^{2\alpha(q)'}}
  \end{align*}
  and $\frac{m}{|m|}(P)$ is a partial isometry.
\end{proof}

We record the following application of the multiplier theorem to resolvent powers.

\begin{corollary}
Let $(M,\g)$ be a closed Riemannian manifold of dimension $n\geq 2$. Let $2\leq q\leq\infty$ and $\gamma>1/2$. Then for any orthonormal system $(u_j)_{j\in J}\subset \Ran\mathbf{1}_{[0,2\lambda]}(P)$ and coefficients $\nu=(\nu_j)_{j\in J}\subset\C$,
  \begin{align*}
    \bigg\|\sum_{j\in J}\nu_j |(-\Delta_{\g}-(\lambda+i)^2)^{-\gamma}u_j|^2\bigg\|_{L^{q/2}(M)}\lesssim_{\gamma} \lambda^{2\delta(q)-2\gamma}\|\nu\|_{\ell^{\alpha(q)}}.
  \end{align*}
\end{corollary}

\begin{proof}
Consider the function 
$m_{\gamma}(\tau)\coloneqq (\tau^2-(\lambda+i)^2)^{-\gamma}\mathbf{1}_{[0,2\lambda]}(\tau).$
Then, since $\gamma>1/2$,
\[
K_q(m_{\gamma})\lesssim \lambda^{\delta(q)-\gamma}
\bigg(\sum_{k\leq 2\lambda}(1+|k-\lambda|)^{-2\gamma}\bigg)^{1/2}
\lesssim_{\gamma} \lambda^{\delta(q)-\gamma}.\qedhere
\]
\end{proof}

\begin{remark}
 The case $\gamma=1$ also follows from \cite[Proposition 3.1]{CueninNguyenSu2025}. The elliptic part of resolvent powers, corresponding to the projection onto $\Ran\mathbf{1}_{[2\lambda,\infty)}(P)$, can be handled by \cite[Proposition 24]{FrankSabin-2017clust}, combined with \cite[Lemma 2.4]{CueninNguyenSu2025}.    
\end{remark}

\section{Proof of Theorem~\ref{thm:main}}\label{sect. Bourgain--Shao--Sogge--Yao argument}
In the following,  let $T=c_0\log\lambda$, where $c_0$ is a small constant to be fixed later. 
  We adapt an argument of Bourgain--Shao--Sogge--Yao \cite{BourgainShaoYaoSogge-2015} to the case of orthonormal systems. Let $q>q_c,$ and $1\leq\beta<\frac{q(n-1)}{2n}$. By duality, to prove ~\eqref{ineq:main} it suffices to show that
  \begin{align}\label{eq:dual main estimate}
    \|W\Pi_{\lambda,T^{-1}}\overline{W}\|_{\mathfrak{S}^{\beta'}}\lesssim_{q,\beta} \lambda^{n(1-\frac{2}{q})-1}(\log\lambda)^{-1}\|W\|^2_{L^{2(q/2)'}}.
  \end{align}
  Let $a$ be a nonnegative Schwartz function on $\R$ such that
  \begin{align}
    a\geq \mathbf{1}_{[-1,1]} \quad\text{and}\quad \supp(\widehat{a})\subset[-1,1].
  \end{align}
  Since $a\geq \mathbf{1}_{[-1,1]}$, it follows that
  \begin{align}
    Wa(T(P-\lambda))\overline{W}\geq W\Pi_{\lambda,T^{-1}} \overline{W}
  \end{align}
  in the quadratic form sense.
  Since for positive compact operators \(0\le B\le A\), one has \(\|B\|_{\mathfrak S^p}\le \|A\|_{\mathfrak S^p}\), \eqref{eq:dual main estimate} would thus follow from
  \begin{align}\label{eq:dual main estimate smooth}
    \|Wa(T(P-\lambda))\overline{W}\|_{\mathfrak{S}^{\beta'}}\lesssim \lambda^{n(1-\frac{2}{q})-1}(\log\lambda)^{-1}\|W\|^2_{L^{2(q/2)'}}.
  \end{align}
  By Fourier inversion, we can write
  \begin{align}
    a(T(P-\lambda))=\frac{1}{2\pi T}\int_{-\infty}^{\infty}\widehat{a}(t/T)\e^{-\I t\lambda}\e^{\I tP}\rd t .
  \end{align}
  Using that $2\cos (tP) = e^{i t P} + e^{-i t P}$ and discarding the smoothing operator $a(-T(P+\lambda))$ arising from the $e^{-i t P}$ term (see Remark \ref{rem:smoothing operator} below), it suffices to consider
  \begin{align}
    A\coloneq\frac{1}{T}\int_{-\infty}^{\infty}\widehat{a}(t/T)\e^{-\I t\lambda}\cos(tP) \rd t
  \end{align}
  and to prove 
  \begin{align}\label{eq:dual main estimate smooth cos}
    \|WA\overline{W}\|_{\mathfrak{S}^{\beta'}}\lesssim \lambda^{n(1-\frac{2}{q})-1}(\log\lambda)^{-1}\|W\|^2_{L^{2(q/2)'}},\quad q>q_c,\quad \beta<\frac{q(n-1)}{2n} .
  \end{align}

  We now split $A$ into a local and global part, $A=A_{\rm loc}+A_{\rm glo}$, where
  \begin{align}
    A_{\rm loc}\coloneq\frac{1}{T}\int_{-\infty}^{\infty}b(t)\widehat{a}(t/T)\e^{-\I t\lambda}\cos(tP) \rd t,\quad
    A_{\rm glo}\coloneq\frac{1}{T}\int_{-\infty}^{\infty}(1-b(t))\widehat{a}(t/T)\e^{-\I t\lambda}\cos(tP) \rd t 
  \end{align}
  and $b\in C_c^{\infty}([-2,2])$ is such that \(b=1\) on \([-1,1]\).

\begin{lemma}[Estimate for the local part]
For $2\leq q\leq\infty$,
  \begin{align}\label{eq:Aloc at q}
    \|WA_{\rm loc}\overline{W}\|_{\mathfrak{S}^{\alpha(q)'}}\lesssim_q \lambda^{2\delta(q)}(\log\lambda)^{-1}\|W\|^2_{L^{2(q/2)'}} ,\quad W\in L^{2(q/2)'}(M),
  \end{align}
  where $\delta(q)$ and $\alpha(q)$ are given by \eqref{eq:delta(q)} and \eqref{eq:alpha-q}, respectively.
\end{lemma}

\begin{proof}
We write $A_{\rm loc}$ as a spectral multiplier,
\begin{align}
    A_{\rm loc}=m_{\rm loc}(P),\quad m_{\rm loc}(\tau)\coloneq  \frac{1}{T}\int_{-\infty}^{\infty}b(t)\widehat{a}(t/T)\e^{-\I t\lambda}\cos(t\tau) \rd t.
  \end{align}
  Integration by parts yields
  \begin{align}\label{eq:bound mloc}
    |m_{\rm loc}(\tau)|\lesssim_N T^{-1}((1+|\lambda-\tau|)^{-N}+(1+|\lambda+\tau|)^{-N}) .
  \end{align}
  Using Theorem~\ref{lemma multiplier} and observing that
  \begin{align}
    K_q(|m_{\rm loc}|^{1/2})
    \lesssim_N  T^{-1/2}\lambda^{\delta(q)}+\lambda^{-N} ,
  \end{align}
  yields the claimed bound.
\end{proof}  

\begin{remark}\label{rem:smoothing operator}
Since 
\[
|a(-T(\tau+\lambda)|\lesssim_N (1+|\lambda+\tau|)^{-N},
\]
the same proof yields
\[
\|W a(-T(P+\lambda)) \overline W\|_{\mathfrak S^{\alpha(q)'}}
\lesssim_N \lambda^{-N}\|W\|^2_{L^{2(q/2)'}}.
\]
\end{remark}

\begin{lemma}[Estimate for the global part]
Under the nonpositive curvature assumption and for $q>q_c$, there exists $\eps(q)>0$ such that
  \begin{align}\label{Aglo at s' with power gain}
    \|WA_{\rm glo}\overline{W}\|_{\mathfrak{S}^{2(q/2)'}}&\lesssim_{q}\lambda^{n(1-\frac{2}{q})-1-\eps(q)}\|W\|^2_{L^{2(q/2)'}} ,\quad W\in L^{2(q/2)'}(M).
  \end{align}
Moreover, for any $\beta<\frac{q(n-1)}{2n}$ there exists $\eps(q,\beta)>0$ such that
    \begin{align}\label{Aglo final at alphaq}
    \|WA_{\rm glo}\overline{W}\|_{\mathfrak{S}^{\beta'}}
    \lesssim_{q,\beta}
    \lambda^{n(1-\frac{2}{q})-1-\eps(q,\beta)}
    \|W\|^2_{L^{2(q/2)'}} ,\quad W\in L^{2(q/2)'}(M).
  \end{align}
\end{lemma}

\begin{proof}[Proof of \eqref{Aglo at s' with power gain}]
Again we write  $A_{\rm glo}$ as a spectral multiplier,
  \begin{align}
    A_{\rm glo}=m_{\rm glo}(P),\quad m_{\rm glo}(\tau)\coloneq  \frac{1}{T}\int_{-\infty}^{\infty}(1-b(t))\widehat{a}(t/T)\e^{-\I t\lambda}\cos(t\tau) \rd t .
  \end{align}
  Integration by parts yields
  \begin{align}\label{eq:bound mglo}
    |m_{\rm glo}(\tau)|\lesssim_N (1+|\lambda-\tau|)^{-N}+(1+|\lambda+\tau|)^{-N}.
  \end{align}
  Using Theorem~\ref{lemma multiplier} and the estimate
  \begin{align}
    K_{q_c}(|m_{\rm glo}|^{1/2})
    \lesssim\lambda^{\frac{1}{q_c}}
  \end{align}
  yields 
  \begin{align}\label{eq:A_glo at q=q_c}
    \|W_1A_{\rm glo}W_2\|_{\mathfrak{S}^{n+1}}\lesssim \lambda^{\frac{n-1}{n+1}}\|W_1\|_{L^{n+1}}\|W_2\|_{L^{n+1}}.
  \end{align}

We recall that $A_{\rm glo}$ depends on $T=c_0\log\lambda$. Given $\eta>0$ we may now fix $c_0$ such that the pointwise kernel estimate 
  \begin{align}\label{L1toLinfty bound Aglo Bourgain--Shao--Sogge--Yao}
    \|A_{\rm glo}\|_{L^{\infty}(M\times M)}\lesssim_{\eta} \lambda^{\frac{n-1}{2}+\eta}
  \end{align}
  holds, where we denoted the kernel of $A_{\rm glo}$ by the same symbol as the operator. This follows from \cite[Eq. (5.7)]{BourgainShaoYaoSogge-2015} and is also implicit in earlier work of Bérard \cite{Berard-1977} and Hassell--Tacy \cite{HassellTacy-2015}. 
  As an immediate consequence of \eqref{L1toLinfty bound Aglo Bourgain--Shao--Sogge--Yao}, we have the Hilbert--Schmidt bound
  \begin{align}\label{eq:A_glo at q=infinity}
    \|W_1A_{\rm glo}W_2\|_{\mathfrak{S}^{2}}\lesssim_{\eta} \lambda^{\frac{n-1}{2}+\eta}\|W_1\|_{L^{2}}\|W_2\|_{L^{2}}.
  \end{align}
We interpolate the bilinear map
\[
  \mathcal T(W_1,W_2)\coloneq W_1A_{\rm glo}W_2 .
\]
By \eqref{eq:A_glo at q=q_c}, we have
\[
  \|\mathcal T(W_1,W_2)\|_{\mathfrak S^{n+1}}
  \lesssim
  \lambda^{\frac{n-1}{n+1}}
  \|W_1\|_{L^{n+1}}\|W_2\|_{L^{n+1}},
\]
while \eqref{eq:A_glo at q=infinity} implies
\[
  \|\mathcal T(W_1,W_2)\|_{\mathfrak S^{2}}
  \lesssim_{\eta}
  \lambda^{\frac{n-1}{2}+\eta}
  \|W_1\|_{L^{2}}\|W_2\|_{L^{2}} .
\]
Hence, by bilinear complex interpolation
(see, e.g., \cite[Theorem 4.4.1]{MR482275}),
for every \(0\le \theta\le 1\),
\begin{align}\label{eq. TW1W2}
  \|\mathcal T(W_1,W_2)\|_{\mathfrak S^{r_\theta}}
  \lesssim_{\eta}
  \lambda^{(1-\theta)\frac{n-1}{n+1}
  +\theta(\frac{n-1}{2}+\eta)}
  \|W_1\|_{L^{p_\theta}}\|W_2\|_{L^{p_\theta}},
\end{align}
where
\begin{align}\label{eq. ptheta rtheta}
  \frac1{p_\theta}=\frac{1-\theta}{n+1}+\frac{\theta}{2},
  \qquad
  \frac1{r_\theta}=\frac{1-\theta}{n+1}+\frac{\theta}{2}.
\end{align}
Here we used \cite[Theorem 4.4.1]{MR482275} with
$A_0^{(1)}=A_0^{(2)}=L^{n+1}$, $A_1^{(1)}=A_1^{(2)}=L^{2}$, $B_0=\mathfrak{S}^{n+1}$, $B_1=\mathfrak{S}^2$ and the fact that $L^p$ and Schatten spaces are interpolation spaces, with
$(L^{p_1},L^{p_2})_{\theta}=L^{p_{\theta}}$ and $(\mathfrak{S}^{p_1},\mathfrak{S}^{p_2})_{\theta}=\mathfrak{S}^{p_{\theta}}$ for $\frac{1}{p_{\theta}}=\frac{1-\theta}{p_1}+\frac{\theta}{p_2}$ and $1\leq p_1,p_2\leq\infty$.

Choosing
\(
  \theta=1-\frac{q_c}{q},
\)
we obtain
\(
  p_\theta=r_\theta=2(q/2)'
\)
in \eqref{eq. ptheta rtheta}.
Therefore, setting \(W_1=W\) and \(W_2=\overline{W}\) in \eqref{eq. TW1W2},
\begin{align}\label{eq:A_glo interpolated}
  \|WA_{\rm glo}\overline W\|_{\mathfrak S^{2(q/2)'}}
  \lesssim_{\eta}
  \lambda^{(n-1)(\frac{1}{2}-\frac{1}{q})+\eta(1-\frac{q_c}{q})}
  \|W\|_{L^{2(q/2)'}}^2.
\end{align}
Fixing $\eta:=\frac{n-1}{4}$, this shows that
 \eqref{Aglo at s' with power gain} holds with
\[
      \eps(q)=\frac{n-1}{4}\Big(1-\frac{q_c}{q}\Big).\qedhere
\]
  \end{proof}

\begin{proof}[Proof of \eqref{Aglo final at alphaq}]
Let \(\chi\in C_c^\infty(\R)\) such that \(0\leq \chi\leq 1\), \(\chi=1\) on \([-1,1]\), and \(\supp\chi\subset [-2,2]\). For $R\in [1,\lambda]$, to be fixed later, define
  \[
    m_{\rm near}^{(R)}(\tau)\coloneq\chi\Big(\frac{\tau-\lambda}{R}\Big)m_{\rm glo}(\tau),
    \qquad
    m_{\rm far}^{(R)}(\tau)\coloneq\Big(1-\chi\Big(\frac{\tau-\lambda}{R}\Big)\Big)m_{\rm glo}(\tau),
  \]
  and
  \begin{equation}\label{def:Anear-far}
    A_{\rm near}^{(R)}\coloneq m_{\rm near}^{(R)}(P),
    \qquad
    A_{\rm far}^{(R)}\coloneq m_{\rm far}^{(R)}(P).
  \end{equation}
Then \(A_{\rm glo}=A_{\rm near}^{(R)}+A_{\rm far}^{(R)}\). We will collect some properties of these operators in Lemma \ref{lem:near_far} below.

We first estimate the near part. In view of \eqref{Aglo at s' with power gain}, we may assume that \(\beta'<2(q/2)'\). Using
  ~\eqref{eq:rank-improvement-Anear-R} with 
 \(s'=2(q/2)'\),  we have
  \begin{align}
    \|WA_{\rm near}^{(R)}\overline{W}\|_{\mathfrak{S}^{\beta'}}
    \lesssim
    (R\lambda^{n-1})^{\frac1{\beta'}-\frac1{2(q/2)'}}
    \|WA_{\rm near}^{(R)}\overline{W}\|_{\mathfrak{S}^{2(q/2)'}}.
  \end{align}
  Since \(A_{\rm near}^{(R)}\) is a spectral multiplier satisfying the same pointwise
  bound~\eqref{eq:bound mglo} as \(A_{\rm glo}\), the estimate
  ~\eqref{eq:A_glo interpolated} applies to \(A_{\rm near}^{(R)}\) as well. Thus, after some straightforward arithmetic simplifications, using $\alpha(q)=\frac{q(n-1)}{2n}$ for $q\geq q_c$, 
  \begin{align*}
    \|WA_{\rm near}^{(R)}\overline{W}\|_{\mathfrak{S}^{\beta'}}
    &\lesssim
    R^{\frac1{\beta'}-\frac1{2(q/2)'}}
    \lambda^{n(1-\frac{2}{q})-1+\eta(1-\frac{q_c}{q})-(n-1)(\frac{1}{\beta}-\frac{2n}{q(n-1)})}
    \|W\|^2_{L^{2(q/2)'}}.
  \end{align*}
We now set 
  \begin{align*}
\eps(q,\beta)\coloneqq \frac{n-1}{4}\Big(\frac{1}{\beta}-\frac{2n}{q(n-1)}\Big),\quad 
\eta(q,\beta)\coloneqq \frac{2\eps(q,\beta)}{1-\frac{q_c}{q}},\quad R(\lambda,q,\beta)\coloneqq \lambda^{\frac{\eps(q,\beta)}{\frac{1}{\beta'}-\frac{1}{2(q/2)'}}}
  \end{align*}
  and fix $\eta=\eta(q,\beta)$, $R=R(\lambda,q,\beta)$ to obtain
  \begin{align}\label{estimate near part}
    \|WA_{\rm near}^{(R)}\overline{W}\|_{\mathfrak{S}^{\beta'}}
    &\lesssim
    \lambda^{n(1-\frac{2}{q})-1-\eps(q,\beta)}
    \|W\|^2_{L^{2(q/2)'}}.
  \end{align}
  
  Using \eqref{eq:Afar-negligible-R} with $N$ sufficiently large yields a better estimate than \eqref{estimate near part} for the far part $A_{\rm far}^{(R)}$. 
Combining the near and far bounds, we obtain \eqref{Aglo final at alphaq}.
\end{proof}

\begin{lemma}\label{lem:near_far} With the above notation, the following properties hold.
     \begin{enumerate}[(i)]
    \item We have
          \begin{align}\label{eq:rank-Anear-R}
            \operatorname{rank} (WA_{\rm near}^{(R)}\overline{W})\lesssim R\,\lambda^{n-1}
          \end{align}
   and if \(1\leq \beta'\leq s'\leq \infty\), then
          \begin{align}\label{eq:rank-improvement-Anear-R}
            \|WA_{\rm near}^{(R)}\overline{W}\|_{\mathfrak{S}^{\beta'}}
            \lesssim
            (R\lambda^{n-1})^{\frac1{\beta'}-\frac1{s'}}
            \|WA_{\rm near}^{(R)}\overline{W}\|_{\mathfrak{S}^{s'}}.
          \end{align}       

    \item For every \(q\geq 2\) and every \(N>0\),
          \begin{align}\label{Kq-mfar-R}
            K_q\big(|m_{\rm far}^{(R)}|^{1/2}\big)
            \lesssim_N
            R^{-N}\lambda^{\delta(q)}.
          \end{align}
          Consequently, Theorem~\ref{lemma multiplier} implies
          \begin{align}\label{eq:Afar-negligible-R}
            \|WA_{\rm far}^{(R)}\overline{W}\|_{\mathfrak{S}^{\alpha(q)'}}
            \lesssim_N
            R^{-N}\lambda^{2\delta(q)}
            \|W\|_{L^{2(q/2)'}(M)}^2.
          \end{align}
  \end{enumerate}
\end{lemma}

\begin{proof}
  Part (i) follows by definition immediately, since \(m_{\rm near}^{(R)}\) is supported where
  \(|\tau-\lambda|\leq 2R\), so
  \[
    \operatorname{rank}A_{\rm near}^{(R)}
    \leq
    \operatorname{rank}\mathbf{1}_{[\lambda-2R,\lambda+2R]}(P)
    \lesssim R\,\lambda^{n-1},
  \]
  where the last inequality follows from the Weyl law \eqref{eq:weyl-law}; the logarithmic improvement is irrelevant here.

  By \eqref{eq:rank-Anear-R} and H\"older,
  \[
    \|WA_{\rm near}^{(R)}\overline{W}\|_{\mathfrak{S}^{\beta'}}
    \lesssim
    (R\lambda^{n-1})^{\frac1{\beta'}-\frac1{s'}}
    \|WA_{\rm near}^{(R)}\overline{W}\|_{\mathfrak{S}^{s'}}.
  \]

  For (ii), on the support of \(m_{\rm far}^{(R)}\) one has \(|\tau-\lambda|\geq R\). Therefore, by the assumption~\eqref{eq:bound mglo},
  \[
    |m_{\rm far}^{(R)}(\tau)|
    \lesssim_N    R^{-N/2}(1+|\lambda-\tau|)^{-N/2}+(1+|\lambda+\tau|)^{-N}.
  \]
  Hence
  \begin{align*}
    K_q\big(|m_{\rm far}^{(R)}|^{1/2}\big)^2
    =
    \sum_{k=0}^\infty
    \sup_{\tau\in[k,k+1)}|m_{\rm far}^{(R)}(\tau)|
    (1+k)^{2\delta(q)} 
    \lesssim_N
    R^{-N/2}\lambda^{2\delta(q)}+\lambda^{-N}.
  \end{align*}
  After replacing \(N\) by \(2N\), this gives~\eqref{Kq-mfar-R}. Inequality \eqref{eq:Afar-negligible-R} follows from Theorem~\ref{lemma multiplier}.
\end{proof}

\section*{Acknowledgements}
 J.-C. C. and X. S. acknowledge support through the Engineering \& Physical Sciences Research Council (EP/X011488/1).
 N. N. N. acknowledges partial support through the European Union through the European Research Council’s Starting Grant FermiMath, grant agreement nr. 101040991.

\bibliographystyle{siam}
\bibliography{biblio.bib}

\end{document}